\documentclass[11pt]{amsart} 
\usepackage{amsfonts,amsbsy,amsmath,amssymb, amsthm}
\usepackage{color}
\usepackage[colorlinks=true, citecolor=Blue, urlcolor = Maroon, linkcolor=Maroon, anchorcolor = Maroon]{hyperref}  
\usepackage[T1]{fontenc}
\usepackage{lipsum}                     
\usepackage{xargs}                      
\usepackage[pdftex,dvipsnames]{xcolor}  
\usepackage[colorinlistoftodos,prependcaption,textsize=footnotesize]{todonotes}
\newcommandx{\unsure}[2][1=]{\todo[linecolor=red,backgroundcolor=red!25,bordercolor=red,#1]{#2}}
\newcommandx{\change}[2][1=]{\todo[linecolor=blue,backgroundcolor=blue!25,bordercolor=blue,#1]{#2}}
\newcommandx{\info}[2][1=]{\todo[linecolor=OliveGreen,backgroundcolor=OliveGreen!25,bordercolor=OliveGreen,#1]{#2}}
\newcommandx{\improvement}[2][1=]{\todo[linecolor=Plum,backgroundcolor=Plum!25,bordercolor=Plum,#1]{#2}}
\newcommandx{\thiswillnotshow}[2][1=]{\todo[disable,#1]{#2}}

\makeatletter
\def\moverlay{\mathpalette\mov@rlay}
\def\mov@rlay#1#2{\leavevmode\vtop{%
   \baselineskip\z@skip \lineskiplimit-\maxdimen
   \ialign{\hfil$\m@th#1##$\hfil\cr#2\crcr}}}
\newcommand{\charfusion}[3][\mathord]{
    #1{\ifx#1\mathop\vphantom{#2}\fi
        \mathpalette\mov@rlay{#2\cr#3}
      }
    \ifx#1\mathop\expandafter\displaylimits\fi}
\makeatother

\newcommand{\scaps}[1]{{\scshape #1}}
\newcommand{\bscaps}[1]{\textsc{\textbf{#1}}}
\newcommand{\slanted}[1]{\slshape{#1}}
\newcommand{\bfm}{\textbf}
\newcommand{\mcal}{\mathcal}

\DeclareMathOperator{\Ex}{{\bf E}}

\DeclareMathOperator{\Var}{{\bf Var}}

\newcommand{\cL}{\mcal{L}}

\newcommand{\cS}{\mcal{S}}

\newcommand{\D}{\mcal{D}}
\newcommand{\T}{\mcal{T}}
\newcommand{\NN}{\mathbb{N}}

\let\eps=\varepsilon
\let\theta=\vartheta
\let\rho=\varrho
\let\phi=\varphi

\newenvironment{theorem}
{
\refstepcounter{equation} 
\vspace{-1ex}
\ \\
\noindent
\begin{it}
\noindent
\bscaps{Theorem~\theequation.}\hspace{-0.5ex}
}
{\end{it} \vspace{1ex}}

\newenvironment{lemma}
{
\refstepcounter{equation} 
\vspace{-1ex}
\ \\
\noindent
\begin{it}
\noindent
\bscaps{Lemma~\theequation.}\hspace{-0.5ex}
}
{\end{it} \vspace{1ex}}

\def\QED{$\blacksquare$}
\def\inQED{$\square$}

\renewenvironment{proof}
{\vspace{1ex}\noindent{\slanted Proof.}\hspace{0.5em}}{\hfill \QED \vspace{1ex}}

\newenvironment{proofof}[1]
{\vspace{1ex}\noindent\scaps{Proof of #1.}\hspace{0.5em}}{\hfill \QED \vspace{1ex}}

\makeatletter
\def\section{\@ifstar\unnumberedsection\numberedsection}
\def\numberedsection{\@ifnextchar[
  \numberedsectionwithtwoarguments\numberedsectionwithoneargument}
\def\unnumberedsection{\@ifnextchar[
  \unnumberedsectionwithtwoarguments\unnumberedsectionwithoneargument}
\def\numberedsectionwithoneargument#1{\numberedsectionwithtwoarguments[#1]{#1}}
\def\unnumberedsectionwithoneargument#1{\unnumberedsectionwithtwoarguments[#1]{#1}}
\def\numberedsectionwithtwoarguments[#1]#2{%
  \ifhmode\par\fi
  \removelastskip
  \vskip 1.7ex\goodbreak
  \refstepcounter{section}%
  \begingroup
  \noindent\leavevmode\Large\bfseries\scshape\normalsize
  \begin{center}\S \thesection.\ #2\end{center} 
  \endgroup
  \addcontentsline{toc}{section}{%
    \protect\numberline{\bfm{\thesection.}}%
    \hspace{2.5ex} #1}%
  }
\def\unnumberedsectionwithtwoarguments[#1]#2{%
  \ifhmode\par\fi
  \removelastskip
  \vskip 1.7ex\goodbreak
  \begingroup
  \noindent\leavevmode\Large\bfseries\scshape\centering 
  \begin{center} #2 \end{center} \par
  \endgroup
  \vskip 2ex\nobreak
  \addcontentsline{toc}{section}{%
    \hspace{1ex} #1}%
  }
\makeatother

\makeatletter
\def\subsection{\@ifstar\unnumberedsubsection\numberedsubsection}
\def\numberedsubsection{\@ifnextchar[
  \numberedsubsectionwithtwoarguments\numberedsubsectionwithoneargument}
\def\unnumberedsubsection{\@ifnextchar[
  \unnumberedsubsectionwithtwoarguments\unnumberedsubsectionwithoneargument}
\def\numberedsubsectionwithoneargument#1{\numberedsubsectionwithtwoarguments[#1]{#1}}
\def\unnumberedsubsectionwithoneargument#1{\unnumberedsubsectionwithtwoarguments[#1]{#1}}
\def\numberedsubsectionwithtwoarguments[#1]#2{%
  \ifhmode\par\fi
  \removelastskip
  \vskip 1.7ex\goodbreak
  \refstepcounter{subsection}%
  \noindent
  \leavevmode
  \begingroup
  \bfseries\normalsize
  \S~\thesubsection\ \bscaps{#2.}
  \endgroup
  \addcontentsline{toc}{subsection}{%
    \hspace{2ex}\protect\numberline{\bfm{\thesubsection.}}%
    \hspace{1ex} #1}%
  }
\def\unnumberedsubsectionwithtwoarguments[#1]#2{%
  \ifhmode\par\fi
  \removelastskip
  \vskip 3ex\goodbreak
  \noindent
  \leavevmode
  \begingroup
  \bfseries\normalsize
  \S~\bscaps{#2.}
  \endgroup
  \addcontentsline{toc}{subsection}{%
    \hspace{1ex} #1}%
  }
\makeatother

\usepackage{ulem}
\newcommand{\Gnp}{G(n,p)}
\usepackage{verbatim}

\begin{document}

\title[Monochromatic Schur triples]{\bscaps{Monochromatic Schur triples in randomly perturbed dense sets of integers}} 
\author{Elad Aigner-Horev} 
\address{Department of Mathematics and Computer Science, Ariel University, Ariel, Israel}
\email{horev@ariel.ac.il}

\author{Yury Person}
\address{Institut f\"ur Mathematik, Technische Universit\"at Ilmenau, 98684 Ilmenau, Germany}
\email{yury.person@tu-ilmenau.de}
\thanks{YP is supported by the Carl Zeiss Foundation.}

\maketitle

\begin{abstract}
Given a dense subset $A$ of the first $n$ positive integers,  we provide a short proof showing that for $p=\omega(n^{-2/3})$ the so-called 
{\sl randomly perturbed} set $A \cup [n]_p$ a.a.s. has the property that any $2$-colouring of it has a monochromatic Schur triple, i.e.\ 
a triple of the form $(a,b,a+b)$. This result is optimal since there are dense sets $A$, for which $A\cup [n]_p$ does not possess this property for $p=o(n^{-2/3})$.
\end{abstract}

\section{Introduction}
The model of randomly perturbed graphs was introduced by 
Bohman, Frieze and Martin in~\cite{BFM03}, 
where they considered for which $p=p(n)$ a Hamilton cycle appears a.a.s.\  if one studies 
the  model $G_\alpha\cup \Gnp$, where $G_\alpha$ is a graph on $n$ vertices with 
minimum degree $\delta(G)\ge \alpha n$ (for a fixed $\alpha>0$) and $\Gnp$ is the usual binomial random graph, i.e.\ each of the 
$\binom{n}{2}$ possible edges appears independently with probability $p$. Their result was 
a discovery of a phenomenon, that it suffices to take $p=C(\alpha)/n$, which is lower by a log-term needed 
for the appearance in the random graph alone, see e.g.~\cite{Bol98}. Further properties were studied 
in~\cite{BFKM04,KST,SV08} and, even more recently, various spanning structures in
 randomly perturbed graphs and hypergraphs were 
investigated in~\cite{BTW17,BHKM18,BDF17, BHKMPP18, BMPP18, DRRS18, HZ18,KKS16,KKS17,MM18}.

The study of Ramsey properties in random graphs was initiated by \L{}uczak, Ruci\'nski and Voigt~\cite{LRV92} and 
the the so-called symmetric edge Ramsey problem was settled completely by R\"odl and Ruci\'nski in a series of papers~\cite{RR93,RR94,RR95}.
Ramsey properties in randomly perturbed graphs were investigated first by  Krivelevich, Sudakov, and Tetali~\cite{KST}, who proved 
 that whenever one perturbs a sufficiently large $n$-vertex dense graph $G$ by adding to it $\omega(n^{2-2/(t-1)})$ edges chosen 
 uniformly at random then the resulting graph almost surely possesses the property that any $2$-colouring of 
 its edges admits either a monochromatic triangle (in the first colour) or a monochromatic copy of $K_t$ (in the second colour), 
 where here $t \geq 3$. Generally, one writes $G\longrightarrow (K_3,K_t)$ to denote this fact (and shortens it to 
 $G\longrightarrow (K_t)_r$ in the symmetric case and $r$ colors).  
 Moreover, the result of~\cite{KST} is asymptotically best possible in terms of the number of random edges added. 
 
In this short note we will study randomly perturbed dense sets of integers. More precisely, we denote by 
$[n]_p$ the model where each of the integers from $[n]:=\{1,2,\ldots,n\}$ is chosen independently with probability $p=p(n)$. 
The model $[n]_p$ itself was thoroughly investigated with respect to extremal properties (e.g.\ Szemer\'edi's theorem~\cite{CG16,KLR96,Schacht}) and Ramsey-type 
properties (partition regularity of equations~\cite{FRS10,GRR96}). 
 
 Given a dense set $A\subseteq [n]$, i.e.\ with at least $\alpha n$ elements for some fixed positive $\alpha$,  
we will study for which $p=p(n)$, the set $A\cup [n]_p$ a.a.s.\ satisfies the property that no matter how one colours $A\cup [n]_p$
with two colours, there will always be a monochromatic Schur triple, i.e.\ 
a triple of integers $x,y,z$ with  $x+y =z$. An old result of Schur~\cite{Schur} states that
 if  $\NN$ is finitely coloured, then there is always such a monochromatic triple. Moreover, the threshold $\Theta(n^{-1/2})$ for the random set 
 $[n]_p$ was determined by Graham, R\"odl and Ruci\'nski~\cite{GRR96} (for two colours) and 
 by Friedgut, R\"odl and Schacht~\cite{FRS10} (for any constant number of colours). 
 Our main result on Schur triples in the model $A\cup [n]_p$ and two colours is as follows.

\begin{theorem}\label{thm:Schur}
For every $\alpha > 0 $ and every $p=\omega(n^{-2/3})$ it holds that 
whenever $A \subseteq [n]$ has $|A| \geq \alpha n $ then a.a.s.\ $A \cup [n]_p$ has the property that in any $2$-colouring of it 
a monochromatic Schur triple appears. 
\end{theorem}

Let us say a few words about the nature of the threshold $\omega(n^{-2/3})$. The point is that, above $n^{-2/3}$, the set $A\cup[n]_p$ 
contains a.a.s.\ some $10$-element set, which, no matter how coloured by red and blue, always contains a monochromatic Schur triple. 
A similar situation occurs for $G_\alpha\cup G(n,p)\longrightarrow (K_3)_2$: it was shown in~\cite{BFKM04} that the threshold 
for the appearance of $K_6$ is then $p=\omega(1/n)$ and noted in~\cite{KST} that, since $K_6\longrightarrow (K_3)_2$, this 
is already the threshold.

The condition $p = \omega(n^{-2/3})$ from Theorem~\ref{thm:Schur} is asymptotically best possible. The expected number of Schur triples in the random set $[n]_p$ is $O(p^3n^2)$; this quantity vanishes whenever $p = o(n^{-2/3})$. A typical set at this density is a.a.s.\ {\em sum-free then} (i.e., has no Schur triples). This in turn implies that the perturbation of $[n/2+1,n]$, namely
$[n/2+1,n] \cup [n]_p$, a.a.s.\ admits a $2$-colouring with no monochromatic Schur triples. Indeed, one may colour the (dense) sum-free set $[n/2+1,n]$ with one colour and all elements (not already in $[n/2+1,n]$) coming from $[n]_p$ with the complementary colour. Moreover, for $r > 2$ colours the colouring in which the set $A:=[n/2+1,1]$ is coloured with one colour and $[n]_p$ with the remaining $r-1$ shows that, asymptotically, one does not gain on probability since the threshold for $[n]_p$ having the property that any $(r-1)$-colouring (for any constant $r\ge 3$) of it admits a monochromatic Schur triple is $\Theta(n^{-1/2})$~\cite{FRS10,GRR96}. 

\section{Proof of Theorem~\ref{thm:Schur}}

We refer to a set of integers as {\em $2$-Schur-Ramsey} if any $2$-colouring of it admits a monochromatic Schur triple. 
For $x,y,d \in [n]$ satisfying $y > x+d$ and $x>d$ we define the $10$-tuple $\cL(x,y,d)\in \NN^{10}$ as follows
\begin{equation}\label{eq:form1}
\cL(x,y,d) : = (x-d,x,x+d,x+2d,y-d,y,y+d, d, y-x-d, y-x).
\end{equation}
Its first four entries form a $4$-term arithmetic progression ($4$AP, hereafter), the next three form a $3$-term arithmetic progression ($3$AP, hereafter), 
and its last three points form a Schur triple denoted $\cS(x,y,d):=(d, y-x-d, y-x)$. 

\begin{lemma}\label{lem:local-reason}
Let $n \geq 1$ be an integer and let $x,y,d \in [n]$ satisfy $y > x+d$, $x>d$. Then 
$\cL(x,y,d)$ is $2$-Schur-Ramsey. 
\end{lemma}

\begin{proof}
The proof proceeds via a simple case analysis. Fix a $2$-colouring of $\cL(x,y,d)$ and assume towards a contradiction that it admits no monochromatic Schur triple. We may assume that $x,y$, and $d$ are not coloured in the same colour; for otherwise the assumption of no monochromatic Schur triple would force a contradiction in the shape of the Schur triple $(y-x,x-d,y-d)$ all coloured with the complimentary colour to the one used for $x,y$, and $d$. 

Nevertheless, two of $x,y$, and $d$ are coloured using the same colour. We consider the three possible cases arising here. Suppose, firstly, that $y$ and $d$ are coloured, say, red, and that $x$ is coloured blue. We may assume that $y-d$ and $y+d$ are both coloured blue in this setting. Then the Schur triple $(y-x-d,x,y-d)$ forces $y-x-d$ to be red and consequently $y-x$ is blue owing to $(y-x-d,d,y-x)$. The triple $(y-x,x-d,y-d)$ mandates that $x-d$ is red and then one concludes that $x+d$ is blue due to $(y-x-d,x+d,y)$. At this point one notices that $(y-x,x+d,y+d)$ is a blue Schur triple, a contradiction.

Suppose, secondly, that $x$ and $d$ are coloured, say, red and that $y$ is coloured blue. We may assume that $x-d$ and $x+d$ are both coloured blue. The triple $(y-x-d,x+d,y)$ implies that $y-x-d$ is red and then $y-x$ is blue due to $(d,y-x-d,y-x)$. Consequently, $y-d$ is red owing to $(y-x,x-d,y-d)$ and then one arrives at $(y-x-d,x,y-d)$ as a red Schur triple, a contradiction. 

Suppose, thirdly (and finally), that $x,y$ are coloured, say, blue and that $d$ is coloured red. Due to $(y-x,x,y)$ we may assume that $y-x$ is coloured red. Next, we have that $y-x-d$ is blue due to $(d,y-x-d,y-x)$, that $y-d$ is red due to $(y-x-d,x,y-d)$, that $x+d$ is red due to $(y-x-d,x+d,y)$, that $y+d$ is blue due to $(y-x,x+d,y+d)$, that $x-d$ is blue due to $(y-x,x-d,y-d)$, and that $x+2d$ is blue due to $(x+d,d,x+2d)$. All this ends with $(y-x-d,x+2d,y+d)$ being a blue Schur triple; contradiction. 
\end{proof}

Facilitating our proof is the following result due to Varnavides~\cite{Var}; which roughly speaking asserts that dense sets in $[n]$ have $\Omega(n^2)$ $4$APs. 
Originally Varnavides' result was phrased for $3$APs, but a standard supersaturation argument generalizes   to arithmetic progressions of any fixed length.  

\begin{theorem}{\em (Varnavides~\cite{Var})}\label{thm:Varnavides}
For every $\alpha > 0 $ there exists an $n_0:= n_0(\alpha)$ and a $g(\alpha) > 0$ such that whenever $n \geq n_0$ and $A \subseteq [n]$ satisfy $|A| \geq \alpha n$ then $A$ contains at least $g(\alpha) n^2$ $4$APs. 
\end{theorem}

Dense sets in $[n]$ have an abundance of $4$APs then. This in turn implies that there are elements in $[n]$ that are repeatedly being used by $4$APs found in the dense set as their common steps; in fact there are many such {\em popular common steps}. Making this precise, define for a set $A \subseteq [n]$ and a parameter $\eps >0$ the set
$$
\D_\eps(A) :=\{d \in [n]: \text{$A$ has $\geq \eps n$ $4$APs with common step $d$}\}
$$
to be the set of common steps {\sl popular} in the sense that each of these participates in at least $\eps n$ $4$APs (in $A$). 

\begin{lemma}\label{lem:popular} {\em (Linearly many popular common steps)}
For every $\alpha >0$ there exists an $\eps := \eps(\alpha) > 0$  such that for $n$ sufficiently large  $|\D_\eps(A)| \geq \eps n$ holds whenever $A \subseteq [n]$ satisfies $|A| \geq \alpha n$. 
\end{lemma}

\begin{proof} Given $\alpha$, let $g(\alpha)$ be as asserted by Theorem~\ref{thm:Varnavides}. Set $\eps:=g(\alpha)/2$ and let $\# 4AP(A)$ denote the number of $4$APs in $A$. 
By Theorem~\ref{thm:Varnavides}, the set $A$ has at least $g(\alpha)n^2$ $4$APs whenever $n$ is sufficiently large. Then owing to 
\[
g(\alpha) n^2 \leq \# 4AP(A) < |\D_\eps(A)|n + n\cdot \eps n
\]
the lemma follows.
\end{proof}

We are now ready to prove Theorem~\ref{thm:Schur}. 
 
\begin{proofof}{Theorem~\ref{thm:Schur}}
By Lemma~\ref{lem:local-reason} it suffices to show that a.a.s. $A \cup [n]_p$ contains $\cL(x,y,d)$ for some $x,y,d \in [n]$
 (with $y > x+d$ and $x>d$). By Theorem~\ref{thm:Varnavides} the set $A$ contains 
 at least $g(\alpha) n^2$ $4$APs. Moreover, Lemma~\ref{lem:popular} asserts that there is an $\eps := \eps(\alpha)$  such that 
$|\D_\eps(A)| \geq \eps n$. 
Given $d \in \D_{\eps}(A)$ and a pair of {\em distinct} $4$APs both with common step $d$, we may write them  parametrically 
as $(x-d,x,x+d,x+2d)$ and $(y'-d,y',y'+d,y'+2d)$ with $y'>x$. Setting $y:=y'+d$, we see that  the Schur triple $\cS(x,y,d)=(d, y-x-d, y-x)$ is well-defined. 

Let $\T$ be the set of all such Schur triples $\cS(x,y,d)$. Since there are at least $\eps n$ choices for $d\in D_\eps(A)$ and $A$ contains at least 
$\eps n$ $4$APs with common difference $d$, there are, for a fixed $d\in D_\eps(A)$, at least $\eps n-1$ different possible values for $y-x$. 
Therefore, the number of such Schur triples is at least $\eps^2n^2/4$, i.e.\ $|\T|\ge \eps^2n^2/4$. 


It remains to argue that a.a.s. $[n]_p$ contains a member of $\T$. The expected number of members of $\T$ captured by $[n]_p$ is at least $\eps^2 p^3n^2/4$.
 We write $X_t$ for the indicator random variable indicating whether $t=\cS(x,y,d) \in \T$ is captured by $[n]_p$ 
 and then put $X := \sum_{t\in \T} X_t$. Using Chebyshev's inequality we arrive at: 
 \[
  \Pr\left[X=0\right]\le \frac{\Var(X)}{(\Ex X)^2}\le \frac{1}{\Ex(X)}+\frac{\sum_{t\neq t' \in \T} \Ex [X_t X_{t'}]}{\Ex(X)^2}.
 \]
The number of pairs $t,t' \in \T$, $t \not= t'$, having a single entry in common is  at most $3n^3$, and each such pair satisfies $\Ex [X_t X_{t'}] \leq p^5$ so in total all such pairs contribute $O(p^5n^3)$ to the sum appearing on the r.h.s.\  above. Next, the number of pairs $t,t' \in \T$, $t \not= t'$, having two entries in common is  at most $2n^2$, with the total contribution of $O(n^2p^4)$. 
Using the lower bound on $\Ex (X)\ge \eps^2 p^3n^2/2$ and $p = \omega(n^{-2/3})$ we have
\[
\Pr \big[ X=0 \big] \leq \frac{1}{\Omega(p^3n^2)} + \frac{O(p^5n^3)}{\Omega(p^6n^4)} +\frac{O(p^4n^2)}{\Omega(p^6n^4)}=o(1),
\]
finishing the proof.
\end{proofof}

\section{Concluding remarks}
In this note we studied the first Ramsey-type question for randomly perturbed sets of integers, i.e.\ the model $A\cup [n]_p$. It would be interesting to determine the thresholds for 
more general partition regular systems of equations. For the threshold results in random sets see the work of Friedgut, R\"odl and Schacht~\cite{FRS10}. 
 
\bibliographystyle{amsplain} 
\bibliography{Schur-Lit.bib}

\end{document}